\documentclass[11pt, a4paper]{amsart}

\usepackage{amsmath, amsthm, amssymb}
\usepackage{txfonts}
\usepackage{mathrsfs}

\theoremstyle{plain}
\newtheorem{thm}{Theorem}[section]
\newtheorem{lem}[thm]{Lemma}
\newtheorem{cor}[thm]{Corollary}
\newtheorem{prop}[thm]{Proposition}

\theoremstyle{definition}
\newtheorem{dfn}[thm]{Definition}

\theoremstyle{remark}

\begin{document}
\title{The Countable Chain Condition for C*-Algebras}
\author[S.~Masumoto]{Shuhei MASUMOTO}
\address[S.~Masumoto]{Graduate~School~of~Mathematical~Sciences, the~University~of~Tokyo}
\email{masumoto@ms.u-tokyo.ac.jp}


\begin{abstract}
In this paper, we introduce the countable chain condition for C*-algebras and 
study its fundamental properties.  
We show independence from $\mathsf{ZFC}$ of the statement that 
this condition is preserved under the tensor products of C*-algebras.  
\end{abstract}

\maketitle

\section{Introduction}\label{sec:intro}
A topological space is said to have \emph{the countable chain condition} (\emph{CCC} for short)
if every family of mutually disjoint nonempty open subsets is countable.  
Any separable space clearly has CCC.  
Conversely, every metric space which has CCC is separable.  

The relation between separability and direct products is simple.  
The direct product of a family of separable spaces are separable 
when its cardinality is less than or equal to $2^\omega$; 
but if the cardinality of the family is greater than $2^\omega$, 
then its direct product can be nonseparable.  
On this point, CCC behaves differently: 
it is irrelevant to the cardinality of the family.  
It is known that the direct product of a family of CCC spaces has CCC 
if for every finite subfamily, its direct product has CCC; 
however, the statement that the direct product of two CCC spaces has CCC 
cannot be proved or disproved in $\mathsf{ZFC}$ \cite[Theorem II.2.24 and Lemma II.4.3]{Kunen}.  

Now we shall restrict our interest to locally compact Hausdorff spaces.  
Let $X$ be a locally compact Hausdorff space and 
$C_0(X)$ be the C*-algebra of the continuous functions on $X$ which vanish at infinity.  
In view of the Gelfand-Naimark theorem, 
$C_0(X)$ contains all the information about the topological structure of $X$.  
In particular, there is a canonical one to one correspondence between 
the open sets of $X$ and the closed ideals of $C_0(X)$, 
and CCC can be reformulated as a condition on the ideal structure of $C_0(X)$,  
whence this condition can be generalized for noncommutative C*-algebras.  
Moreover, since $C_0(X \times Y)$ is canonically isomorphic to $C_0(X) \otimes C_0(Y)$, 
the discussion on the relation between CCC and direct products yields information 
about the ideal structure of tensor products of C*-algebras.  
In this way, we prove the following theorems in this paper: 

\begin{thm}\label{thm:main1}
The minimal tensor product of a family of unital CCC C*-algebras has CCC 
if for every finite subfamily, its minimal tensor product has CCC. 
\end{thm}

\begin{thm}\label{thm:main2}
Martin's Axiom, $\mathsf{MA}(\omega_1)$, 
implies that any minimal tensor product of unital CCC C*-algebras has CCC.  
\end{thm}

The precise definition of CCC for C*-algebras is introduced in Section \ref{sec:dfn}.  
In Section \ref{sec:MA} Martin's Axiom, 
which is known to be independent from $\mathsf{ZFC}$, is explained.  
Here it is also verified that the negation of the Suslin Hypothesis, 
which is another independent statement explained in Section \ref{sec:MA}, 
implies the opposite conclusion of Theorem \ref{thm:main2}.  
We prove Theorems \ref{thm:main1} and \ref{thm:main2} in Section \ref{sec:tensor}.  
Combining this fact with Theorem \ref{thm:main1}, 
we conclude that the statement that tensor products of CCC C*-algebras has CCC is 
independent from $\mathsf{ZFC}$.

\section{A definition of CCC for C*-algebras}\label{sec:dfn}
\begin{dfn}\label{dfn:CCC_C*-algebra}
Two nonzero ideals in a C*-algebra are said to be \emph{orthogonal} 
if their intersection is the zero ideal.  
A C*-algebra has the \emph{countable chain condition} (CCC)
if any family of nonzero mutually orthogonal ideals is countable.  
\end{dfn}

Note that if $\mathcal{I}, \mathcal{J}$ are ideals in a C*-algebra, then 
$\mathcal{I} \cap \mathcal{J}$ coincides with $\overline{\mathcal{I}\mathcal{J}}$, 
whence they are orthogonal if and only if $\mathcal{I}\mathcal{J} = 0$.  

We shall begin with verifying that this definition is a generalization of CCC for topological spaces.  
Recall that a topological space has CCC if any family of nonempty mutually disjoint open subsets 
is countable.  

\begin{prop}\label{prop:CCC_equiv_CCC}
Let $X$ be a locally compact Hausdorff space.  
Then $C_0(X)$ has CCC as a C*-algebra if and only if $X$ has CCC as a topological space.  
\end{prop}

\begin{proof}
Suppose first that $X$ has CCC and let $\{\mathcal{I}_\lambda\}_{\lambda \in \Lambda}$ 
be a family of nonzero mutually orthogonal ideals in $C_0(X)$.  
We can take an element $f_\lambda\in \mathcal{I}_\lambda$ of norm $1$ for each $\lambda$.  
Set $U_\lambda = \{ x \in X ~|~ |f_\lambda (x)| > 1/2 \}$.  
Then $\{U_\lambda\}_{\lambda \in \Lambda}$ is a family of nonempty mutually disjoint open 
subsets of $X$, whence $\#\Lambda \leq \omega$.  Thus, $C_0(X)$ has CCC by definition.  

If $X$ admits an uncountable family $\{U_\lambda\}_{\lambda \in \Lambda}$ of 
nonempty mutually disjoint open sets, 
then $\{ C_0(U_\lambda) \}_{\lambda \in \Lambda}$
is an uncountable family of nonzero mutually orthogonal ideals of $C_0(X)$.  
Therefore, $C_0(X)$ does not have CCC.  
\end{proof}

The following easy proposition characterizes CCC.  
Note that a von Neumann algebra is said to be \emph{$\sigma$-finite} 
if it admits no uncountable family of mutually orthogonal projections.  

\begin{prop}\label{prop:alg_con}
\leavevmode
\begin{enumerate}
\item Let $\mathcal{A}$ be a C*-algebra.  Then $\mathcal{A}$ has CCC 
if and only if there exists no family $\{ a_\lambda \}_{\lambda \in \Lambda}$ of
nonzero elements such that $a_\lambda \mathcal{A} a_\mu = 0$ for $\lambda \neq \mu$.  
\item A von Neumann algebra has CCC if and only if its center is $\sigma$-finite.  
\end{enumerate}
\end{prop}

\begin{proof}
\noindent
(i) Suppose that there is an uncountable family $\{ a_\lambda \}_{\lambda \in \Lambda}$ of
nonzero elements such that $a_\lambda \mathcal{A} a_\mu = 0$ for $\lambda \neq \mu$.   
For each $\lambda \in \Lambda$, 
let $\mathcal{I}_\lambda = \overline{\mathcal{A} a_\lambda \mathcal{A}}$ be 
the ideal generated by $a_\lambda$.  
Then $\{\mathcal{I}_\lambda\}_{\lambda \in \Lambda}$ is an uncountable family of 
nonzero mutually orthogonal ideals, so $\mathcal{A}$ does not have CCC.  

Conversely, assume that $\mathcal{A}$ does not have CCC and   
let $\{ \mathcal{I}_{\lambda \in \Lambda} \}$ be an uncountable family of nonzero 
mutually orthogonal ideals.  Taking nonzero $a_\lambda \in \mathcal{I}_\lambda$ 
for each $\lambda$, we obtain $a_\lambda \mathcal{A} a_\mu = 0$ for $\lambda \neq \mu$ 
because $\mathcal{I}_\lambda \mathcal{I}_\mu = 0$.  \\
(ii) Let $\mathcal{I}_1, \mathcal{I}_2$ be ideals of a von Neumann algebra $\mathcal{M}$.  
Then it can be easily verified that $\mathcal{I}_1\mathcal{I}_2 = 0$ if and only if 
$\bar{\mathcal{I}_1}^{\sigma{\rm w}}\bar{\mathcal{I}_2}^{\sigma{\rm w}} = 0$, 
where $\bar{\mathcal{I}_i}^{\sigma{\rm w}}$ denotes the 
$\sigma$-weak closure of $\mathcal{I}_i$.  
Now $\bar{\mathcal{I}_i}^{\sigma{\rm w}}$ is of the form $\mathcal{M}z_i$ 
for a central projection $z_i$, and the two ideals are orthogonal if and only if 
these projections are orthogonal.  
Therefore, $\mathcal{M}$ has CCC if and only if there is no uncountable family of 
nonzero mutually orthogonal projections, that is, $\sigma$-finite.   
\end{proof}

\begin{prop}
A separable C*-algebra has CCC.  
\end{prop}

\begin{proof}
Suppose that $\mathcal{A}$ does not have CCC, 
and $\{\mathcal{I}_\lambda\}_{\lambda \in \Lambda}$ be 
an uncountable family of nonzero mutually orthogonal ideals.  
If $h_\lambda \in \mathcal{I}_\lambda$ is a positive element of norm $1$, 
then it follows by functional calculus that $\|h_\lambda-h_\mu\| = 1$.  
If we denote by $U_\lambda$ the open ball of radius $1/2$ centered at $h_\lambda$, 
then $\{U_\lambda\}_{\lambda \in \Lambda}$ is 
an uncountable family of mutually disjoint open subsets.  
Hence, $\mathcal{A}$ is not separable.  
\end{proof}

An ideal of a CCC C*-algebra clearly has CCC.  
Also, it can be easily verified that an extension of a CCC C*-algebra by a CCC C*-algebra has CCC.  
On the other hand, a quotient of a CCC C*-algebra does not necessarily have CCC.  
Indeed, let $\beta\mathbb{N}$ be the Stone-\v{C}ech compactification of $\mathbb{N}$.  
It has CCC because it is separable.  However, the Stone-\v{C}ech remainder 
$\beta\mathbb{N} \setminus \mathbb{N}$ does not have CCC because 
there exists an almost disjoint family of $2^\omega$ subsets of $\omega$ 
\cite[Theorem II.1.3]{Kunen}.  
Therefore, $C(\beta\mathbb{N} \setminus \mathbb{N})$ does not have CCC, 
although it is the quotient of the CCC C*-algebra 
$C(\beta\mathbb{N}) \simeq \ell^\infty$ by $C_0(\mathbb{N}) \simeq c_0$.  

Since $C(\beta\mathbb{N} \setminus \mathbb{N})$ can be obtained as the inductive limit of
$\ell^\infty \stackrel{\varphi}{\to} \ell^\infty \stackrel{\varphi}{\to} \cdots$, 
where $\varphi \colon \ell^\infty \to \ell^\infty$ is defined by $\varphi(f)(n) = f(n+1)$, 
it also follows that inductive limits of CCC C*-algebras do not necessarily have CCC.  
On this direction, what we can prove is the following: 

\begin{prop}\label{prop:CCC_union}
Let $\mathcal{A}$ be a C*-algebra and $\kappa$ be an infinite cardinal number 
with its cofinality not equal to $\omega_1$.  
If there is an increasing sequence $\{\mathcal{A}_\alpha\}_{\alpha < \kappa}$ of 
CCC C*-subalgebras such that 
$\overline{\bigcup_{\alpha < \kappa} \mathcal{A}_\alpha} = \mathcal{A}$, 
then $\mathcal{A}$ has CCC.  
\end{prop}

To prove this proposition, we use the lemma below.  
A proof can be found in \cite[Lemma III.4.1]{Davidson}.  

\begin{lem}\label{lem:union_lemma}
Let $\mathcal{A}$ be a C*-algebra and $\{\mathcal{A}_\alpha\}$ be a directed set of 
subalgebras with its union dense in $\mathcal{A}$.  
If $\mathcal{I}$ is an ideal of $\mathcal{A}$, then it is obtained as the closure of 
the union of $\{\mathcal{I} \cap \mathcal{A}_\alpha\}$.  
\end{lem}

\begin{proof}[Proof of Proposition \ref{prop:CCC_union}]
Assume that there is an uncountable family $\{\mathcal{I}_\lambda\}_{\lambda < \omega_1}$ of 
nonzero mutually orthogonal ideals of $\mathcal{A}$.  
For each $\lambda$, set 
\[
\beta_\lambda = 
\min\{\alpha < \kappa ~|~ \mathcal{I}_\lambda \cap \mathcal{A}_\alpha \neq 0\}, 
\]
which exists by Lemma \ref{lem:union_lemma}, and write $\beta = \sup_\lambda \beta_\lambda$.  
If $\beta < \kappa$ holds, then $\{\mathcal{A}_\beta \cap \mathcal{I}_\lambda\}_\lambda$ is 
an uncountable family of nonzero mutually orthogonal ideals, which contradicts to the fact that 
$\mathcal{A}_\beta$ has CCC.  On the other hand, if $\beta = \kappa$, 
then the cofinality of $\kappa$ is $\omega$, whence there is an unbounded increasing sequence 
$\{\gamma_n\}_{n < \omega}$ in $\kappa$.  
Now the set $S_n$ of $\lambda < \omega_1$ with $\beta_\lambda < \gamma_n$ is at most 
countable for each $n$, whence $\omega_1 = \#\bigl(\bigcup_n S_n \bigr) \leq \omega$, 
a contradiction.  
\end{proof}

We close this section by looking at the relation between CCC and von Neumann tensor products.  
The following proposition, together with results in Section \ref{sec:tensor}, reveals 
that the situation in the von Neumann algebra setting differs from that in case of C*-algebras.   

\begin{prop}\label{prop:vN_CCC}
Let $\mathcal{M}$ and $\mathcal{N}$ be CCC von Neumann algebras.  
Then the tensor product $\mathcal{M} \bar{\otimes} \mathcal{N}$ of $\mathcal{M}$ and 
$\mathcal{N}$ as a von Neumann algebra also has CCC.  
\end{prop}

\begin{proof}
We shall denote by $\mathcal{Z}(\mathcal{M}), \mathcal{Z}(\mathcal{N})$ 
and $\mathcal{Z}(\mathcal{M} \bar{\otimes} \mathcal{N})$ the centers of 
$\mathcal{M}, \mathcal{N}$ and $\mathcal{M} \bar{\otimes} \mathcal{N}$ respectively.   
Recall that $\mathcal{Z}(\mathcal{M} \bar{\otimes} \mathcal{N})$ coincides with 
$\mathcal{Z}(\mathcal{M}) \bar{\otimes} \mathcal{Z}(\mathcal{N})$ 
\cite[Corollary IV.5.11]{Takesaki1}.  Hence it suffices to show that the tensor product of 
two abelian $\sigma$-finite von Neumann algebras is also $\sigma$-finite.  
To see this, note that every abelian von Neumann algebra is of the form 
$L^\infty(\mu)$ for some Radon measure $\mu$ \cite[Theorem III.1.18]{Takesaki1},  
and it is $\sigma$-finite if and only if $\mu$ is $\sigma$-finite.  
Since $L^\infty(\mu) \bar{\otimes} L^\infty(\nu)$, 
being canonically isomorphic to $L^\infty(\mu \otimes \nu)$, is $\sigma$-finite 
if $L^\infty(\mu)$ and $L^\infty(\nu)$ are both $\sigma$-finite, 
the conclusion follows.  
\end{proof}

A compact Hausdorff space is a \emph{stonean space} if the closure of every open set is open.  
Suppose that $X$ is a stonean space and $\mu$ is a Borel measure on it.  
If for any increasing family $\{f_i\} \in C_\mathbb{R}(X)$ with $\sup f_i = f \in C_\mathbb{R}(X)$ 
the equality $\mu(f) = \sup \mu(f_i)$ holds, then $\mu$ is said to be \emph{normal}.  
A stonean space is called a \emph{hyperstonean space} if 
for any nonzero positive $f \in C_\mathbb{R}(X)$ there exists a normal Borel measure $\mu$ such that 
$\mu(f) > 0$.  It is known that if $X$ is hyperstonean, then $C(X)$ is a von Neumann algebra, 
and every abelian von Neumann algebra is of this form \cite[Theorem III.1.18]{Takesaki1}.  
Combining this fact with the preceding proposition, we obtain the following result.  

\begin{cor}
The direct product of two hyperstonean CCC spaces has CCC.  
\end{cor}

\begin{proof}
Let $X, Y$ be hyperstonean CCC spaces.  
It follows from Proposition \ref{prop:vN_CCC} that the von Neumann tensor product 
$C(X) \bar{\otimes} C(Y)$ has CCC, and  
$C(X) \otimes C(Y)$, which is isomorphic to $C(X \times Y)$, 
is a C*-subalgebra of $C(X) \bar{\otimes} C(Y)$.  
By Proposition \ref{prop:alg_con}, it can be easily verified that
 any C*-subalgebra of commutative CCC C*-algebra has CCC, 
whence $X \times Y$ has CCC.  
\end{proof}

\section{Martin's axiom and Suslin's hypothesis}\label{sec:MA}
In this section, we introduce two statements which are known to be independent from $\mathsf{ZFC}$.  
Complete treatise for these statements can be found in \cite{Kunen} or \cite{Jech}.  

The first statement is Martin's axiom.  
We shall introduce some definitions related to partially ordered sets 
in order to express this axiom in a simple form.  

\begin{dfn}
Let $P$ be a nonempty partially ordered set.  
Two elements $p, q \in P$ are \emph{incompatible} if there is no $r \in P$ 
with $r \leq p$ and $r \leq q$.  
If there is no uncountable family of mutually incompatible elements in $P$, 
then $P$ is said to have the \emph{countable chain condition} (CCC).  
\end{dfn}

As is easily verified, a C*-algebra has CCC if and only if its nonzero ideals form 
a CCC partially ordered set, where the order is defined by inclusion.  
Similarly, a nonempty topological space has CCC if and only if 
the set of nonempty open subsets has CCC as a partially ordered set.  

\begin{dfn}
Let $P$ be a partially ordered set.  
\begin{enumerate}
\item A subset $D \subset P$ is \emph{dense} if for any $p \in P$ there is $q \in D$ with $q \leq p$.  
\item A nonempty subset $F \subset P$ is called a \emph{filter} on $P$ if it satisfies the following: 
\begin{enumerate}
\item if $p, q$ are in $F$, then there exists $r \in F$ with $r \leq p$ and $r \leq q$; 
\item if $p \in F$ and $q \geq p$, then $q \in F$.  
\end{enumerate}
\end{enumerate}
\end{dfn}

Suppose that $P$ is a nonempty partially ordered set and 
fix the topology generated by subsets of the form $\{ q \in P ~|~ q \leq p \}$ for $p \in P$.  
Then $P$ has CCC if and only if $P$ has CCC as a topological space, 
and $D \subset P$ is dense if and only if it is dense as a topological subspace.  

Now we shall see the exact statement of Martin's axiom $\mathsf{MA}$.  
Let $\kappa$ be a cardinal number.  

\bigskip
\begin{description}
\item[$\mathsf{MA}(\kappa)$] 
If $P$ is a nonempty CCC partially ordered set and $\{ D_\alpha \}_{\alpha \in \kappa}$ 
is a family of dense subsets in $P$, then there exists a filter $F$ on $P$ such that 
$F \cap D_\alpha$ is not empty for all $\alpha$.  
\item[$\mathsf{MA}$] $\mathsf{MA}(\kappa)$ holds for any $\kappa$ with 
$\omega \leq \kappa < 2^\omega$.  
\end{description}
\bigskip

It is known that $\mathsf{MA}(\omega)$ holds (the Rasiowa-Sikorski lemma) and 
$\mathsf{MA}(2^\omega)$ does not hold in $\mathsf{ZFC}$, 
whence the Continuum Hypothesis $\mathsf{CH}$ trivially implies $\mathsf{MA}$. 
On the other hand, $\mathsf{MA}$ is indeed consistent with $\mathsf{ZFC}+\neg\mathsf{CH}$.  
In particular,$\mathsf{ZFC}+\mathsf{MA}(\omega_1)$ is consistent.  

The other statement we use in this paper is Suslin's Hypothesis $\mathsf{SH}$.  
This hypothesis is related to characterization of the real line as an ordered set.  
Note that a totally ordered set with the following properties is order-isomorphic to the real line: 
\begin{enumerate}
\item unbounded; there does not exist minimum nor maximum element.  
\item dense; there is an element between any two elements.  
\item complete; every nonempty bounded subset has a supremum and an infimum.  
\item separable; there is a countable subset which is dense with respect to the usual order topology.  
\end{enumerate}

\begin{dfn}
Let $S$ be a totally ordered set which is unbounded, dense and complete.  
Then $S$ is called a \emph{Suslin line} if it is nonseparable but CCC as a topological space, 
where its topology is the usual order topology generated by open intervals.  
\end{dfn}

\bigskip
\begin{description}
\item[$\mathsf{SH}$] There does not exist a Suslin line.  
\end{description}
\bigskip

In other words, $\mathsf{SH}$ claims that separability in the characterization of the real line above 
can be replaced by CCC.  It is known that the diamond principle $\diamondsuit$, 
which is a consequence of the axiom of constructibility $\mathbf{V}=\mathbf{L}$, 
implies $\neg \mathsf{SH}$ \cite{Jensen}.  
On the other hand, $\mathsf{MA(\omega_1)}$ implies $\mathsf{SH}$, 
whence $\mathsf{SH}$ is independent from $\mathsf{ZFC}$.  

\begin{prop}
A Suslin line is a locally compact space.  
\end{prop}

\begin{proof}
It suffices to show that every bounded closed interval is compact.  
This can be verified by seeing that a proof for the Heine-Borel theorem can be applied to Suslin lines.  

Given an open covering $\{U_\lambda\}_{\lambda \in \Lambda}$ of 
a bounded closed interval $[a, b]$, we shall prove that $[a, b]$ can be covered by 
finitely many $U_\lambda$'s.  
Note that we may assume each $U_\lambda$ is an open interval.  

Let $X$ be the set of all $x \in [a, b]$ such that $[a, x]$ can be covered by 
finitely many $U_\lambda$'s.  Then $X$ is not empty because $a$ is in $X$, 
and so $\sup X$ exists by completeness.  
It suffices to show that $\sup X$ belongs to $X$ and coincides with $b$.  
For this, take $\lambda_0 \in \Lambda$ such that $\sup X$ is in $U_{\lambda_0}$.  
Then $X \cap U_{\lambda_0}$ contains some element, say $c$.  
Now $[a, c]$ can be covered by finitely many $U_\lambda$'s, 
and $[c, \sup X]$ is included in $U_{\lambda_0}$, so $\sup X$ is in $X$.    
Also, for any $x \in U_{\lambda_0}$, the interval $[a, x]$ can be covered by 
finitely many $U_\lambda$'s, whence $\sup X$ must coincide with $b$.  
\end{proof}

The following proposition is from \cite[Lemma II.4.3]{Kunen}.  
For the sake of completeness, we include the proof.  

\begin{prop}
If $S$ is a Suslin line, then $S \times S$ does not have CCC.  
\end{prop}

\begin{proof}
By transfinite induction, we shall take $a_\alpha, b_\alpha, c_\alpha \in S$ 
for $\alpha < \omega_1$ so that 
\begin{enumerate}
\item $a_\alpha < b_\alpha < c_\alpha$; 
\item $b_\beta \notin (a_\alpha, c_\alpha)$ for $\beta < \alpha$.  
\end{enumerate}
This can be carried over because for each $\alpha < \omega_1$, 
the set $\{b_\beta ~|~ \beta < \alpha\}$, being countable, is not dense in $S$.  
Setting $U_\alpha := (a_\alpha, b_\alpha) \times (b_\alpha, c_\alpha)$, 
we obtain an uncountable family $\{U_\alpha\}_{\alpha < \omega}$ of 
nonempty mutually disjoint open sets in $S \times S$.  
\end{proof}

\begin{cor}
$\neg \mathsf{SH}$ implies the existence of a unital commutative 
CCC C*-algebra $\mathcal{A}$ such that $\mathcal{A} \otimes \mathcal{A}$ does not have CCC.  
\end{cor}

\begin{proof}
Let $S$ be a Suslin line and consider the one point compactification $S^+$ of $S$.  
Since $S^+$ contains $S$ as a dense subspace, it is a CCC space.  
On the other hand, $S^+ \times S^+$ does not have CCC because it contains $S \times S$.  
Now $\mathcal{A} = C(S^+)$ is a unital commutative CCC C*-algebra, 
but $\mathcal{A} \otimes \mathcal{A}$, being isomorphic to $C(S^+ \times S^+)$, 
does not have CCC.  
\end{proof}

\section{Tensor products}\label{sec:tensor}
Here we shall prove Theorems \ref{thm:main1} and \ref{thm:main2}.  
For the first theorem, we need the following combinatorial lemma, 
which is known as the $\Delta$-system lemma.  
A proof can be found in any standard textbook on set theory in which the method of forcing is dealt with.  

\begin{lem}[$\Delta$-system lemma]
Every uncountable family of finite sets includes an uncountable subfamily 
whose pairwise intersection is constant.  
\end{lem}

\begin{proof}[Proof of Theorem \ref{thm:main1}]
Let $\{\mathcal{A}_i\}_{i \in I}$ be a family of unital C*-algebras such that 
for every finite $J \subset I$, the minimal tensor product $\bigotimes_{i \in J} \mathcal{A}_i$ has CCC.  
We shall prove that $\mathcal{A}:= \bigotimes_{i \in I} \mathcal{A}_i$ also has CCC.  

Suppose that, contrary to our claim, there exists an uncountable family 
$\{\mathcal{I}_\lambda\}_{\lambda \in \Lambda}$ of nonzero mutually orthogonal ideals 
in $\mathcal{A}$.  By Proposition \ref{lem:union_lemma}, we can find a finite subset 
$J_\lambda \subset I$ for each $\lambda \in \Lambda$ such that 
$\mathcal{I}_\lambda \cap \bigotimes_{i \in J_\lambda} \mathcal{A}_i \neq 0$.  
By the $\Delta$-system lemma, we may assume that there exists a set $R$ such that 
$J_\lambda \cap J_\mu = R$ for any $\lambda \neq \mu$.  

Since the tensor products are minimal, 
$\mathcal{I}_\lambda \cap \bigodot_{i \in J_\lambda} \mathcal{A}_i$ 
is not zero for each $\lambda$, where $\bigodot_{i \in J_\lambda} \mathcal{A}_i$ 
is the algebraic tensor products of $\mathcal{A}_i$'s.  
Take nonzero $f_\lambda \in \mathcal{I}_\lambda \cap \bigodot_{i \in J_\lambda} \mathcal{A}_i$ 
for each $\lambda$.  If $R$ is empty, then $f_\lambda f_\mu \neq 0$ for $\lambda \neq \mu$, 
which contradicts with the assumption that $\mathcal{I}_\lambda$ and $\mathcal{I}_\mu$ are 
orthogonal to each other.  Therefore, $f_\lambda$ is of the form 
$\sum_k g_\lambda^k \otimes h_\lambda^k$, 
where $g_\lambda^k$ is in $\bigotimes_{i \in R} \mathcal{A}_i$ and 
$\{h_\lambda^k\}_k$ is a linearly independent set in 
$\bigotimes_{i \in J_\lambda \setminus R} \mathcal{A}_i$.  
If $\lambda \neq \mu$, then the equality $\mathcal{I}_\lambda \mathcal{I}_\mu = 0$ implies that 
$g_\lambda^k a g_\mu^l = 0$ for all $k, l$ and $a \in \bigotimes_{i \in R} \mathcal{A}_i$.  
Since for each $\lambda$ there exists $k$ with $g_\lambda^k \neq 0$, 
it follows that $\bigotimes_{i \in R} \mathcal{A}_i$ does not have CCC by 
Proposition \ref{prop:alg_con}, which is a contradiction.  
Therefore, $\bigotimes_{i \in I} \mathcal{A}_i$ has CCC.  
\end{proof}

\begin{cor}
Every minimal tensor product of unital separable C*-algebras has CCC.  
\end{cor}

Next, we shall prove the second theorem.  
For this, we use the following lemma.  

\begin{lem}\label{lem:FIP}
Suppose that $\mathcal{A}$ is a CCC C*-algebra and 
$\{\mathcal{I}_\alpha\}_{\alpha < \omega_1}$ be a family of its ideals.  
Then $\mathsf{MA}(\omega_1)$ implies that there exists an uncountable subfamily of 
the ideals which has the finite intersection property.  
\end{lem}

\begin{proof}
Set $\mathcal{J}_\alpha := \overline{\sum_{\gamma < \alpha} \mathcal{I}_\gamma}$.  
Then $\mathcal{J}_\alpha$ is a transfinite decreasing sequence of ideals of $\mathcal{A}$.  
We shall first show that there exists $\alpha_0$ such that 
$\mathcal{J}_\beta$ is an essential ideal of $\mathcal{J}_{\alpha_0}$ for all $\beta > \alpha_0$.  
Suppose the contrary. Then we can find an transfinite increasing sequence 
$\{\beta_\delta\}_{\delta < \omega_1} \subset \omega_1$ such that the inclusion 
$\mathcal{J}_{\beta_{\delta+1}} \subset \mathcal{J}_{\beta_\delta}$ is not essential.  
In other words, there exists a nonzero ideal $\mathcal{K}_{\beta_\delta}$ of 
$\mathcal{J}_{\beta_\delta}$ such that 
$\mathcal{K}_{\beta_\delta} \cap \mathcal{J}_{\beta_{\delta+1}} = 0$.  
Now $\{\mathcal{K}_{\beta_\delta}\}_{\delta < \omega_1}$ is an uncountable family of 
mutually orthogonal ideals in $\mathcal{A}$, which is a contradiction.  

Next, let $P$ be the set of nonzero ideals in $\mathcal{J}_{\alpha_0}$.  
Then $P$ has CCC as a partially ordered set, because an ideal of a CCC C*-algebra has CCC.  
For each $\beta > \alpha_0$, we set 
\[
D_\beta = \{ p \in P ~|~ p \subset \mathcal{I}_\gamma \text{ for some } \gamma \geq \beta \}  
\]
and claim that this is dense in $P$.  
To prove this, take an arbitrary $q \in P$.  
Then $q' := q \cap \mathcal{J}_\beta$ is not zero by the definition of $\alpha_0$.  
Here, $\mathcal{J}_\beta$ is approximated by 
$\sum_{\gamma \in S} \mathcal{I}_\gamma$, where $S \subset ]\,\beta, \omega_1 [$ is finite.  
By \cite[Theorem3.1.7]{Murphy}, $\sum_{\gamma \in S} \mathcal{I}_\gamma$ is 
norm closed for each $S$, whence we can use Lemma \ref{lem:union_lemma} to conclude that 
$q'$ is the inductive limit of $\{q \cap \sum_{\gamma \in S} \mathcal{I}_\gamma\}_S$, 
and so there exists $\gamma > \beta$ with $q \cap \mathcal{I}_\gamma \neq 0$.  
Since $q \cap \mathcal{I}_\gamma$ is clearly in $D_\beta$, 
it follows that $D_\beta$ is dense, as desired.  

Now let $F$ be a filter on $P$ such that $F \cap D_\beta$ is not empty for all $\beta$, 
whose existence is guaranteed by $\mathsf{MA}(\omega_1)$.  
Then $\{ \mathcal{I}_\alpha ~|~ \exists p \in F, \ p \subset \mathcal{I}_\alpha \}$ 
has the finite intersection property, and this is uncountable because 
the condition $F \cap D_\beta \neq \varnothing$ for each $\beta$ implies that 
the set of all $\alpha$ such that $\mathcal{I}_\alpha \supset p$ for some $p \in F$ is unbounded 
in $\omega_1$.  This completes the proof .  
\end{proof}

\begin{proof}[Proof of Theorem \ref{thm:main2}]
By Theorem \ref{thm:main1}, it suffices to show that 
if $\mathcal{A}$ and $\mathcal{B}$ have CCC, then $\mathcal{A} \otimes \mathcal{B}$ has CCC.  
Assume that, on the contrary, there exists a family $\{\mathcal{I}_\alpha\}_{\alpha < \omega_1}$ 
of nonzero mutually orthogonal ideals in $\mathcal{A} \otimes \mathcal{B}$.  
Then there exist nonzero ideals $\mathcal{J}_\alpha \subset \mathcal{A}$ and 
$\mathcal{K}_\alpha \subset \mathcal{B}$ with
$\mathcal{J}_\alpha \odot \mathcal{K}_\alpha \subset \mathcal{I}_\alpha$, 
by \cite[Lemma 2.12 (ii)]{Blanchard-Kirchberg}.  
Here, by the preceding lemma, we may assume that 
$\{\mathcal{J}_\alpha\}_\alpha$ and $\{\mathcal{K}_\alpha\}_\alpha$ satisfy 
the finite intersection property.  
Then, $\mathcal{I}_\alpha \cap \mathcal{I}_\beta$ contains 
$(\mathcal{J}_\alpha \cap \mathcal{J}_\beta) \otimes 
(\mathcal{K}_\alpha \cap \mathcal{K}_\beta) \neq 0$, which is a contradiction.  
Therefore, $\mathcal{A} \otimes \mathcal{B}$ has CCC, as expected.  
\end{proof}

\section{Concluding remarks and problems}
Let $\mathcal{A}$ be a C*-algebra.  
By $\rm{Prim}(\mathcal{A})$, 
we shall denote the primitive spectrum of $\mathcal{A}$.  
(For the definition and elementary properties of primitive spectra, see \cite[Chapter 4]{Pedersen}.)  
It can be easily verified that $\mathcal{A}$ has CCC if and only if 
$\rm{Prim}(\mathcal{A})$ has CCC as a topological space, 
and Lemma \ref{lem:FIP} is obtained as a corollary of \cite[Lemma II.2.23]{Kunen}.  
Here, we may replace $\rm{Prim}(\mathcal{A})$ by the prime spectrum $\rm{prime}(\mathcal{A})$, 
because the topologies of these spaces are isomorphic as partially ordered sets.  

In \cite{Wulfsohn}, it is proved that $\rm{Prim}(\mathcal{A} \otimes \mathcal{B})$ is
homeomorphic to $\rm{Prim}(\mathcal{A}) \times \rm{Prim}(\mathcal{B})$ 
provided that either $\mathcal{A}$ or $\mathcal{B}$ is type I.  
Also, in \cite[Proposition 2.17]{Blanchard-Kirchberg}, 
one can find various conditions for $\rm{prime}(\mathcal{A} \otimes \mathcal{B})$ to be 
homeomorphic to $\rm{prime}(\mathcal{A}) \times \rm{prime}(\mathcal{B})$.  
In these cases, Theorem \ref{thm:main2} follows from the corresponding fact for topological spaces 
\cite[Theorem II.2.24]{Kunen}.  


One problem is whether Theorem \ref{thm:main1} and Theorem \ref{thm:main2} can be 
generalized to non-minimal tensor products.  
Since any tensor product has the minimal tensor product as its quotient, 
it depends on whether the kernel of the quotient map, 
which is difficult to be investigated, has CCC.  

Another problem lies in the definition of CCC.  
In this paper we have defined CCC in terms of ideals, 
whence this condition is trivial for simple C*-algebras.  
In order to avoid this phenomenon, 
we can use hereditary C*-algebras in place of ideals: 
we shall say two hereditary C*-subalgebras $\mathcal{A}$ and $\mathcal{B}$ 
are orthogonal to each other if $\overline{\mathcal{A}\mathcal{B}} = 0$; 
a C*-algebra has \emph{strong} CCC 
if there is no uncountable family of nonzero mutually orthogonal hereditary C*-subalgebras.  
Then we can prove the following in the same way as in section \ref{sec:dfn}.  
\begin{itemize}
\item Strong CCC implies CCC.  
\item C*-subalgebras of a strong CCC C*-algebra have strong CCC.  
\item An extension of a strong CCC C*-algebra by a strong CCC C*-algebra has strong CCC.  
\item A von Neumann algebra has strong CCC if and only if it is $\sigma$-finite, 
so tensor products of two strong CCC von Neumann algebras have strong CCC.  
\end{itemize}
It is expected that conclusions similar to the main theorems of this paper are true, 
but the author could not prove this.  

\bigskip
{\bf Acknowledgement.} 
The author gratitudes to Professor Yasuyuki Kawahigashi, who is my adviser, 
for several helpful comments on this paper.  
The author also expresses his thanks to Professor Ilijas Farah, 
Professor Eberhard Kirchberg, Professor George Elliott, Professor Takeshi Katsura, 
Alessandro Vignati, Yuki Arano and Yosuke Kubota for many stimulating conversations.  
This work was supported by the Program for Leading Graduate Schools, MEXT, Japan.

\end{document}